\documentclass[10pt,a4paper]{amsart}
\usepackage{booktabs}
\usepackage{amscd,amssymb,amsopn,amsmath,amsthm,graphics,amsfonts,enumerate,verbatim,calc}
\usepackage[dvips]{graphicx}

\usepackage{mathpazo}
\usepackage{color}
\usepackage{latexsym}
\usepackage{amsthm,amsfonts,amssymb,mathrsfs}
\usepackage{rotating}
\usepackage[leqno]{amsmath}
\usepackage{xspace}
\usepackage[all]{xy}
\usepackage{longtable}
\headsep=1cm \numberwithin{equation}{section}
\newtheorem{theorem}{Theorem}[section]
\newtheorem{lemma}[theorem]{Lemma}
\newtheorem{notation}[theorem]{Notation}
\newtheorem{proposition}[theorem]{Proposition}
\newtheorem{corollary}[theorem]{Corollary}

\theoremstyle{definition}
\newtheorem{definition}[theorem]{Definition}
\theoremstyle{remark}
\newtheorem{remark}[theorem]{Remark}
\newtheorem{fact}[theorem]{Fact}
\newtheorem{example}[theorem]{Example}
\newtheorem{observation}[theorem]{Observation}
\newtheorem{discussion}[theorem]{Discussion}
\newtheorem{question}[theorem]{Question}

\newtheorem{acknowledgement}{Acknowledgement}

\newcommand{\rr}{\operatorname{r}}
\newcommand{\D}{\operatorname{d}}
\newcommand{\hdeg}{\operatorname{hdeg}}
\newcommand{\Ass}{\operatorname{Ass}}

\newcommand{\soc}{\operatorname{soc}}

\newcommand{\UFD}{\operatorname{UFD}}

\newcommand{\rad}{\operatorname{rad}}

\newcommand{\pd}{\operatorname{p.dim}}

\newcommand{\rank}{\operatorname{rank}}

\newcommand{\mm}{\mathfrak m}

\newcommand{\HG}{\operatorname{hdeg}}\newcommand{\SP}{\operatorname{sp}}
\DeclareMathOperator{\Ext}{Ext}\DeclareMathOperator{\gr}{gr}
\DeclareMathOperator{\Soc}{Soc}\DeclareMathOperator{\edim}{edim}
\DeclareMathOperator{\depth}{depth}

\newcommand{\Hom}{\operatorname{Hom}}
 
\DeclareMathOperator{\ord}{ord}

\newcommand{\emb}{\operatorname{emb dim}}
\newcommand{\e}{\operatorname{e}}
\newcommand{\Char}{\operatorname{char}}

\newcommand{\coker}{\operatorname{coker}}
\newcommand{\HH}{\operatorname{H}}

\newcommand{\m}{\operatorname{m}}

\newcommand{\fm}{\frak{m}}
\newcommand{\fp}{\frak{p}}
\newcommand{\fq}{\frak{q}}
\newcommand{\fa}{\frak{a}}

\newcommand{\fn}{\frak{n}}

\begin{document}

\author[]{mohsen asgharzadeh}

\address{}
\email{mohsenasgharzadeh@gmail.com}

\title[ ]
{on the Rees,     Dilworth and Matsumura number}

\dedicatory{}
\begin{abstract}
In this paper we study three numerical invariants of local rings: the Rees number, the Dilworth number, and the Matsumura number. These quantities measure, respectively, the minimal lengths of certain cyclic modules, the maximal numbers of generators of ideals, and a parameter-based analogue in higher dimension. We construct a non-exact local complete-intersection ring over $\mathbb{F}_2$, thereby providing a negative answer to a question of Matsumura on the exactness of zero-dimensional complete intersections. We also present a family of artinian Gorenstein local rings over a field of characteristic zero that are exact, making progress on a question of Watanabe. Finally, answering a further question of Matsumura, we introduce and study the Matsumura number and the class of homologically exact rings, which extend the notion of exactness beyond the one-dimensional case.
\end{abstract}

\maketitle

\section{Introduction}

Rings for which there exists a uniform bound on the minimal number of
generators of all their ideals form a particularly interesting class.
Sally proved that such rings necessarily have Krull dimension at most one
(see \cite{sal}). Here, $\ell(-)$ denotes the length function and
$\mu(-)$ the minimal number of generators.

Let $(R,\fm)$ be a local ring of Krull dimension at most one. For any
ideal $\fa$ and any element $x \in \fm$, one has
$\mu(\fa) \leq \ell(R/xR)$. In particular, define
\[
\rr(R) := \inf\{\ell(R/xR) : x \in \fm\}
\quad \text{and} \quad
\D(R) := \sup\{\mu(\fa) : \fa \subseteq \fm\}.
\]
Then $\rr(R)$ provides an upper bound for $\D(R)$. These numerical
invariants were introduced by Watanabe in the zero-dimensional case and
by Matsumura in the one-dimensional case; they are called the
\emph{Rees number} and the \emph{Dilworth number}, respectively. In
particular, one always has $\D(R) \leq \rr(R)$ (see also \cite{trung}).

The local ring $(R,\fm)$ is said to be \emph{exact} if
$\D(R) = \rr(R)$. This terminology originates from Matsumura, although it
has not become standard in the literature; the property is also known as
the \emph{maximum--minimum property} (see, for instance, \cite{i}).
Matsumura asked the following question in \cite[Question~A]{M}:

\begin{question}\label{he}
	Are zero-dimensional complete intersections exact?
\end{question}

Using both an elementary computation and a more modern approach, we
construct a non-exact local complete-intersection ring over
$\mathbb{F}_2$. In addition, we present a third approach to the question,
with emphasis on the role of the characteristic.

Section~3 addresses a related question posed by Watanabe:

\begin{question}\label{1.4}
	Can one construct an Artinian Gorenstein local ring that is not
	exact over a field of characteristic zero?
\end{question}
We deal with this by presenting several situation of Artinian Gorenstein local rings 
over a field of characteristic zero that are exact. For instance, in   \ref{thm:nine}  (resp. \ref{3}), we show it when $\ell(R)<10$ (resp. $\fm^3=0$), see (resp. $\fm^3=0$).
We next introduce another invariant of a local ring, defined by
\[
\m(R) := \inf\{\ell(R/\underline{x}R) : \underline{x} \text{ is a
	parameter sequence}\},
\]
and refer to it as the \emph{Matsumura number}. In \cite[Question~C]{M},
Matsumura posed the following:

\begin{question}\label{mat2}
	\begin{enumerate}[i)]
		\item Is $\m(R)$ an interesting integer?
		\item Is there a higher-dimensional analogue of exactness?
	\end{enumerate}
\end{question}

In Section~4, we establish some basic properties of $\m(R)$. Drawing on
an idea from Gunston's thesis, we then introduce a new class of
Noetherian local rings, termed \emph{homologically exact} rings. These
coincide with exact rings in dimension one, and their class strictly
contains that of Buchsbaum rings. This generalization removes the
dimension restriction appearing in a result of Ikeda.

Throughout, all rings are assumed to be commutative, Noetherian, and
local, with maximal ideal $\fm$. All modules are finitely generated
unless otherwise stated. Our primary reference is Matsumura's classical
text \cite{Mat}.

\section{A complete-intersection ring which is not exact }

In this section, $R$ is the ring $\mathbb{F}_2[X,Y,Z]/(X^2,Y^2,Z^2)$.
 The ring $R=\bigoplus_{i=0}^3R_i$ is both graded and local.
In particular, any element $a$ of $R $ is of the form $a=a_0+a_1+a_2+a_3$ where $a_i\in R_i$.

\subsection{An elementary approach to Question \ref{he} }

\begin{lemma}\label{12}
Let  $a\in R_2+R_3$ and $0\neq b\in R_1$. Then $(0:_Ra+b)\nsubseteqq \fm^2$.
\end{lemma}

\begin{proof}
 We bring the following claim:
\begin{enumerate}
\item[Claim] A) Let  $r\in R$ be such that $r_0=0$. Then $r^2=0$. Indeed, $(R_2)^2=  (R_3)^2=0$.
It remains to assume that $r$ is homogeneous and  is of degree one. The claim is clear from  the above table.
\end{enumerate}
 Suppose on the contradiction that $(0:_Ra+b)\subseteq \fm^2$. From Claim A), one has $(a+b)^2=0$. Thus, $a+b\in(0:_Ra+b)\subseteq \fm^2$. Since $a\in R_2$  we have $b\in \fm^2$, a contradiction.
\end{proof}

The following elementary yet crucial fact plays a central role throughout this paper:

\begin{fact}\label{main}(Watanabe, Matsumura,  Ikeda; see \cite[Theorem 2]{M})
Let $(A, \fm)$ be a
local Noetherian ring, let $M$ be a finitely generated A-module with $\dim M \leq 1$, let $N\subset M$ and  $a \in \fm$. Then  $\mu(N)\leq\ell(M/aM)$. Furthermore, $\mu(N)=\ell(M/aM)$ if and only if the following  claims hold:
\begin{enumerate}
\item[a)] $\ell(M/aM)<\infty$, and $\dim (M/N) = 0$.
\item[b)] $(0 :_M a) \subset N$, and
$\fm N = aN$.
\end{enumerate}
\end{fact}

\begin{example}Adopt the above notations. 
Then $\rr(R)=4$.
\end{example}

\begin{proof}
	In light of \cite[Example 1.4]{i}, we observe that $\D(R) = 3$ \footnote{This can be proven as a high-school exercise, though it is tedious, at least for me.}. Recall that $\D(R) \leq \rr(R)$. Therefore, $\rr(R) \geq 3$. Since $\ell(R/xR) = 4$, it suffices to show that
	\[
	\inf\{\ell(R/aR): a \in \fm\} \neq 3.
	\]
	Recall that $\mu(\fm^2) = \D(R)$. We apply Fact \ref{main} for $N := \fm^2$ and $M := R$. We aim to show that $(0 : r) \nsubseteq \fm^2$ for all $r \in \fm$. Without loss of generality, assume $r \neq 0$ and $r_0 = 0$. Suppose first that $r$ is homogeneous. To this end, we state the following three claims, depending on the degree of $r$.
	
	\begin{enumerate}
		\item[Claim A)] Let $r \in R_1$. Then $(0 :_R r) \nsubseteq \fm^2$. \\
		Indeed, by Claim A) from Lemma \ref{12}, we have $r^2 = 0$. Thus, $r \in (0 :_R r) \setminus \fm^2$.
		
		\item[Claim B)] Let $r \in R_2$. Then $(0 :_R r) \nsubseteq \fm^2$. \\
		Suppose $r := xy + yz + zx$. Then
		\[
		(x+y)r = x^2y + xyz + zx^2 + xy^2 + y^2z + xzy = 2xyz = 0.
		\]
		Thus, $(x+y) \in (0 :_R r) \nsubseteq \fm^2$. Now, suppose $r := xy + yz$. Then $yr = 0$, implying $y \in (0 :_R r) \nsubseteq \fm^2$. By symmetry, the same claim holds for $xy + xz$ and $yz + xz$. Similarly, for $r := xy$, we also get $yr = 0$, and thus $y \in (0 :_R r) \nsubseteq \fm^2$. The desired claim is verified for all elements of $R_2$.
		
		\item[Claim C)] Let $r \in R_3$. Then $(0 :_R r) \nsubseteq \fm^2$. \\
		Indeed, we know that $R_i = 0$ for all $i > 3$. Therefore, $R_1 r = 0$, which implies $R_1 \subset (0 :_R r) \nsubseteq \fm^2$.
	\end{enumerate}
	
	Thus, the proof in the homogeneous case follows. Now, let $r = r_1 + r_2 + r_3$. By Lemma \ref{12}, we may assume that $r_1 = 0$, i.e., $r = r_2 + r_3$. Since $r$ is not homogeneous, both $r_2$ and $r_3$ are nonzero. By the proof of Claim B), there exists a nonzero $b \in R_1$ such that $r_2 b = 0$. Recall that $R_1 R_3 = 0$, since $R_i = 0$ for all $i > 3$. In particular, $b R_3 = 0$, so $b \in (0 : r_2 + r_3)$. Combining this with the fact that $b \notin \fm^2$, we conclude that
	\(
	(0 : r_2 + r_3) \nsubseteq \fm^2.
	\)
\end{proof}

\begin{corollary}\label{co}
The ring $R$ is a zero-dimensional local  complete-intersection ring which is not exact.
\end{corollary}
\subsection{Second approach to Question \ref{he}}
One can prove that $\rr(R) \geq 4$ by employing the machinery of \textit{Lefschetz} properties. Let us recall the following comments of Watanabe:

\begin{discussion} \label{dis}
	Let $(A, \fm)$ be a graded Artinian algebra over a field $K$.
	\begin{enumerate}
		\item[i)] We say that $A$ has the \it{weak Lefschetz property} if there exists an element $\ell \in A_1$ such that the multiplication map $\times \ell : A_i \to A_{i+1}$ has maximal rank for every $i$.
		
		\item[ii)] The \it{Hilbert function} $(h_0, \ldots, h_c)$ is said to be \it{unimodal} if there exists an index $j$ such that:
		\[
		h_0 \leq h_1 \leq \cdots \leq h_j \geq h_{j+1} \geq \cdots \geq h_c.
		\]
		
		\item[iii)] (See \cite[Proposition 3.5]{wat3}) Suppose $A$ has a unimodal Hilbert function, and let $\ell \in A_1$. Then:
		\[
		\ell \text{ is a weak Lefschetz element for } A \Longleftrightarrow \sup\{\dim A_i\} = \ell(A/\ell A).
		\]
		
		\item[iv)] (See \cite[Lemma 2.2]{i}) We have:
		\[
		\rr(A) = \inf\{\ell(A/rA) : r \in A_1\}.
		\]
		
		\item[v)] (See \cite[Theorem 13]{iw}) Let $A := \frac{K[X_1, \ldots, X_n]}{(X_1^{d_1}, \ldots, X_n^{d_n})}$. Then the maximum and minimum elements of the set $\{I \trianglelefteq A \mid \mu(I) = \D(A)\}$ are powers of the maximal ideal. In particular, $\sup\{\mu(\fm^i)\} = \D(A)$.
		
		\item[vi)] Note that it may happen that $\sup\{\mu(\fm^i)\} \neq \D(A)$; see \cite[Example 4.4]{i}.
	\end{enumerate}
\end{discussion}

Now, we are ready to present the second Proof of Corollary \ref{co}.
To this end, we claim that $\rr(R) \geq 4$. Indeed, one can verify that $R$ has no Lefschetz element (see, e.g., \cite[Remark 2.9]{wja}). The Hilbert function of $R$ is unimodal. In view of Discussion~\ref{dis}(iii), we have:
\[
\sup\{\dim R_i\} \neq \inf\{\ell(R/\ell R): \ell \in R_1\}.
\]
By Discussion~\ref{dis}(iv), it follows that:
\[
\rr(R) = \inf\{\ell(R/rR): r \in R_1\}.
\]
In light of Discussion~\ref{dis}(v), we also have $\D(R) = \sup\{\dim R_i\}$. Since $\D(R) \leq \rr(R)$, we conclude:
\[
\rr(R) > \sup\{\dim R_i\} = 3.
\]
Thus, $\rr(R) \geq 4$, as desired.

\subsection{Third approach to Question \ref{he}}

\begin{lemma}[A monomial chain bound]\label{lem:chain}
	Let \(A=k[\boldsymbol x]/J\) be an Artinian monomial algebra. If its
	standard monomial basis, ordered by divisibility, is covered by \(w\)
	chains, then \(\D(A)\le w\).
\end{lemma}
\begin{proof}
	For any ideal \(I\subseteq A\), take an initial monomial ideal with
	respect to a monomial order. Its minimal monomial generators form an
	antichain, so number at most \(w\). Lift them to elements of \(I\).
	Leading-term reduction shows that the lifts generate \(I\): every
	nonzero remainder would have an initial monomial in the initial ideal
	and could be reduced further. The basis is finite, so reduction
	terminates. This argument applies to nonhomogeneous ideals as well.
\end{proof}

\begin{proposition}\label{prop:char2}
	For any field \(k\), let
	\(A=k[x,y,z]/(x^2,y^2,z^2)\). Then \(A\) is an Artinian local
	complete intersection and
	\[
	\D(A)=3,\qquad
	\rr(A)=\begin{cases}4&\operatorname{char}k=2,\\3&\operatorname{char}k\ne2.
	\end{cases}
	\]
\end{proposition}
\begin{proof}
	Its eight standard monomials have the following three-chain cover:
	\[
	1<x<xy<xyz,\qquad y<yz,\qquad z<xz.
	\]
	Lemma~\ref{lem:chain} gives \(d(A)\le3\), while \(\mu(\mm)=3\).
	The defining sequence of squares is regular in the polynomial ring,
	so the quotient is a complete intersection.
	
	In characteristic two, every \(a\in\mm\) satisfies \(a^2=0\).
	Consequently multiplication by \(a\) on the eight-dimensional vector
	space \(A\) has image contained in its kernel, so its rank is at most
	four. Therefore \(\ell(A/aA)\ge4\) for every \(a\in\mm\). Taking
	\(a=x\) gives colength four. This proves \(\rr(A)=4\) and treats all
	nonhomogeneous elements at once.
	
	If the characteristic is not two, set \(\lambda=x+y+z\). Eliminating
	\(z\) gives
	\[
	A/\lambda A\cong
	k[x,y]/(x^2,y^2,(x+y)^2)
	=k[x,y]/(x^2,xy,y^2),
	\]
	of length three. The inequalities \(3=\D(A)\le \rr(A)\le3\) finish the proof.
\end{proof}

Thus the obstruction in the paper's example is characteristic two,
not the finiteness of \(\mathbb F_2\). The same counterexample works
over every infinite field of characteristic two. It does not settle
the characteristic-zero complete-intersection question.

\begin{observation}\label{prop:wlp}
	Let \(R\) be an Artinian local ring. If its standard graded algebra
	\(A=\gr_{\mm}(R)\) has WLP, then
	\[
	\D(R)=\rr(R)=\SP(R):=\sup\{\mu(\fm^n).
	\]
\end{observation}
\begin{proof}
	Write \(h_i=\dim_k A_i\), and choose a weak Lefschetz element
	\(\lambda\in A_1\). WLP in a standard graded algebra implies
	unimodality: once a multiplication map is surjective, the corresponding
	graded component of \(A/(\lambda)\) is zero, so all its later
	components are zero as well. Consequently
	\[
	\ell(A/\lambda A)=\sum_i\max\{h_i-h_{i-1},0\}=\max_i \{h_i\},
	\quad h_{-1}=0.
	\]
	Lift \(\lambda\) to \(x\in\mm\). There is a graded surjection
	\(A/(\lambda)\twoheadrightarrow\gr(R/xR)\), since the initial ideal
	of \(xR\) contains \(\lambda\). Hence
	\[
	\SP(R)\le \D(R)\le \rr(R)\le\ell(R/xR)
	\le\ell(A/\lambda A)=\SP(R).
	\]
\end{proof}

\section{Dealing with Question \ref{1.4}}

\begin{lemma}\label{pro}
Let $R$ be  an artinian local ring such that $\ell(R)\leq\mu(\fm)+2$. Then $R$ is exact.	
	\end{lemma}

\begin{proof} The case $\ell(R)=\mu(\fm)$ never happens, because $\mu(\fm)=\ell(\fm/\fm^2)<\ell(R)$. Suppose first that $\ell(R)=\mu(\fm)+1$. Let $x\in\fm$ be nonzero.
Then $$\mu(\fm)\leq\sup\{\mu(\fa)\}=\D(R)\leq\rr(R)=\inf\{\ell(R/rR)\}\leq\ell(R/xR)\leq\ell(R)-1=\mu(\fm).$$
Therefore, $\D(R)=\rr(R)=\mu(\fm)$. So, $R$ is exact.

In the case \(\ell(R)=\edim(R)+2\), one has \(\ell(\mm^2)=1\) and
\(\mm^3=0\). Choose \(x,y\in\mm\) with \(xy\ne0\), which is possible
because \(\mm^2\ne0\). Then \(x\notin\mm^2\) and \(xR\) has length at
least two. Hence
\(\ell(R/xR)\le\edim(R)\), and the opposite inequalities
\(\edim(R)\le \D(R)\le \rr(R)\) finish the proof. By definition, $R$ is  exact.
	\end{proof}

\begin{lemma}\label{m2}
	Let $R$ be an artinian local ring such that $\fm^2=0$. Then $R$ is exact.	
\end{lemma}

\begin{proof} We may assume that $\fm$ is nonzero. First note that $\mu(\fm)=\ell(\fm/\fm^2)=\ell(\fm)$.
	In view of  $0\to \fm \to R\to R/\fm\to 0$ we get that
	$\mu(\fm)=\ell(\fm)=\ell(R)-1$.  The claim is now clear from
	Lemma \ref{pro}.
\end{proof}

\begin{proposition}\label{3}
	Let $R$ be an artinian Gorenstein local ring such that $\fm^3=0$. Then $R$ is exact.	
\end{proposition}

\begin{proof} By Lemma \ref{m2}, we may assume that $\fm^2$ is nonzero. Note that $\mu(\fm^2)=\ell(\fm^2/\fm^3)=\ell(\fm^2)$. Since
	$R$ is Gorenstein  its socle is one-dimensional.
Due to $\fm^3=0$, we have $\fm^2\subset \soc(R)$. Consequently, $1\leq\ell(\fm^2)\leq\ell(\soc(R))=1$.
 Thus,
	$\mu(\fm^2)=\ell(\fm^2)=1$. In view of  $0\to \fm^2 \to \fm\to \fm/\fm^2\to 0$ we get that
	$\mu(\fm)=\ell(\fm)-\ell(\fm^2)=\ell(\fm)-1$.  Also, we apply the exact sequence  $0\to \fm  \to R\to R/\fm \to 0$ to see
  $\ell(\fm)=\ell(R)-1.$  Therefore, $\mu(\fm)=\ell(\fm)-1=\ell(R)-2.$ It remains to apply
	Lemma \ref{pro}.
\end{proof}
	
The Gorenstein assumption  in the above proposition  is important:
\begin{example}\label{e9m}
Let $R:=\frac{\mathbb{C}[X,Y,Z,W]}{ (X^2,Y^2,Z^2,W^2,XY,ZW)}$. In view of \cite[Page 160]{M} $R$ is not exact. Let us
present a proof: It has Hilbert function \((1,4,4)\) and standard basis
\(1,x,y,z,w,xz,xw,yz,yw\). The chains
\[
1<x<xz,\quad y<yw,\quad z<yz,\quad w<xw
\]
give \(\D(B)\le4\) by Lemma~\ref{lem:chain}; equality follows from
\(\mu(\mm)=4\).
For \(\lambda=ax+by+cz+dw\), multiplication \(B_1\to B_2\), in the
displayed bases, has matrix
\[
M_\lambda=
\begin{pmatrix}c&0&a&0\\d&0&0&a\\0&c&b&0\\0&d&0&b\end{pmatrix}.
\]
For \(\lambda\ne0\), the vector \((a,b,-c,-d)\ne0\) is in its
kernel, so \(\rank M_\lambda\le3\). At \(a=b=c=d=1\) a three-by-three
minor is a unit, so the rank is three over every field.
Every element of \(\mm\) has the form \(f=\lambda+h\), with \(h\in B_2\).
If \(\lambda\ne0\), the image of \(1\) has a nonzero degree-one part,
the image of \(B_1\) lies in \(B_2\), and \(B_2f=0\). Thus
\(\rank(\times f)=1+\rank M_\lambda\le4\). If \(\lambda=0\), the rank
is at most one. Hence every principal quotient has length at least
five, and \(x+y+z+w\) attains five. We obtain
\[
\D(B)=4<5=\rr(B),\qquad \dim_k\Soc(B)=4.
\]
Note that  $\fm^3$ is generated by degree three monomials in $\{x,y,z,w\}$. They are zero by the relations $x^2=y^2=z^2=w^2=xy=zw=0$. So,
$\fm^3=0$.\end{example}

\begin{notation} \label{not}Let $(R,\fm,k)$ be a local  artinian $k$-algebra. There is $\ell$ such that $\fm^{\ell+1}=0$ and $\fm^{\ell}\neq0$.
 We set $h_i:=\dim_k(\frac{\fm^i}{\fm^{i+1}})$. By the Hilbert function of $R$
we mean $\HH(R):=(1,h_1,\ldots,h_{\ell})$.
\end{notation}

\begin{fact}\label{low}(Watanabe)
	Let $(R,\fm,k)$ be an artinian  local ring. The following assertions hold.
	\begin{enumerate}
		\item[i)] If $\emb(R)=1$, then $R$ is exact.
		\item[ii)] If $\emb(R)=2$ and $k$ is infinite, then $\D(R)=\rr(R)$.
\item[iii)] If $\rr(R)\leq3$ and $k$ is infinite, then $\D(R)=\rr(R)$.
\item[iv)] One has $\rr(R)\leq \rr(\gr_{\fm}(R))$\end{enumerate}
\end{fact}

\begin{proof}	i) Let $x$ be such that $\fm=(x)$. Then $\mu(\fm)=\ell (R/xR)$. Thus, $\D(R)=\rr(R)=1$.

ii) This is in \cite[Theorem 4.2]{i}.

iii) If $\D(R)\leq2$, by ii) we get the claim. We may assume that $\D(R)=3$. Recall that $\D(R)\leq\rr(R)$. So, $\D(R)=\rr(R)=3$.

iv)  This is in \cite[Proposition 2]{wat2}.
\end{proof}

Following Sally \cite{sal2}, a local artinian ring  is said to be \textit{stretched} if $\fm^2$ is a principal ideal.

\begin{proposition} \label{st}
Any  stretched Gorenstein algebra is exact.
\end{proposition}

\begin{proof}
We are going to use the associated graded ring   $\gr_{\fm}(R)$. This is a graded ring and is of zero dimension (see \cite[Theorem 13.9]{Mat}). Keep the notation be as of  Notation \ref{not}. By Macaulay's theorem  $h_i=1$ for all  $2\leq i\leq \ell$
(Sally  \cite[Theorem 1.1(ii)]{sal2} proved this without any use of Macaulay's theorem).
Thus,  $\HH(R)=(1,n,1,\ldots,1)$. Any nonzero map
from (resp. to) a 1-dimensional vector space is injective (resp. surjective). Conclude by this that  $\gr_{\fm}(R)$ has a weak Lefschtez element.
In the light of Discussion  \ref{dis}, $\rr(\gr_{\fm}(R))=n$. By  Fact \ref{low}(iv), $\rr(R)\leq \rr(\gr_{\fm}(R))$. Therefore,
$$n=\mu(\fm)\leq\D(R)\leq\rr(R)\leq\rr(\gr_{\fm}(R))=n.$$So, $R$ is exact.
\end{proof}

\begin{proposition}\label{als}
Any almost stretched Gorenstein algebra $(A,\fm)$ over an algebraically closed field of characteristic $0$ is exact.
\end{proposition}

\begin{proof}
Due to Fact \ref{low} we  may assume that $n:=\emb(A)\geq3$.
Let us clarify the assumption: There is  a regular local ring $(R,\fn)$ such that $A = R/I$. By the almost stretched
we mean that $\mu(\fm^2)=2$. Then $\HH_A(i)=2$  for $2\leq i\leq t$ and $\HH_A(i)=1$ for all $t+1\leq i\leq s$.
We recall the following classification result of Elias and Valla (see \cite[Theorem 4.1]{ev}):
There is a  basis
$Y_1,\ldots,Y_n$ of $\fn$ and $b\in R$ such that $I$ is generated by the following four classes of functions:

\begin{enumerate}
\item[i)] $Y_iY_j$ with $1\leq i < j \leq n$ and $(i, j)\neq (1, 2)$,
\item[ii)] $Y^2_j - d_jY^s_1$ with $3\leq j \leq n$,
\item[iii)] $Y^2_2 - b Y_1 Y_2 -cY^{s-t+1}_1$,
\item[iv)] $Y^t_1Y^2$,
\end{enumerate}where $\{d_j,c\}$ are invertible.
Note that $\fm^2=(y_iy_j)_{i,j}$. From the first three relations we see that $\fm^2=(y_1^2,y_1y_2)$. Let $\xi:=y_1+y_3$. Then $y_1\xi=y_1^2$ and $y_2\xi=y_1y_2$.
Thus, $\fm\xi=\fm^2$. Clearly, $(0:\xi)\subset\fm$.
In the light of Fact  \ref{main}, $\ell(A/ \xi A)=\mu(\fm)=n$. Therefore, $n\leq\D(A)\leq\rr(A)\leq n$, which is what we want to prove.
\end{proof}

\begin{corollary}\label{77}
Let $(R,\fm,k)$ be a Gorenstein artinian local  $k$-algebra, where $k=\overline{k}$ is of zero-characteristic. If $\ell(R)\leq 7$, then $R$ is exact.
\end{corollary}

\begin{proof}
First,  assume that $\ell(R)= 6$. In this case $\mu(\fm)\leq3$ (see Lemma \ref{pro}). The case $\mu(\fm)\leq2$ is in Fact \ref{low}.
Then $\mu(\fm)=\ell(\fm/ \fm^2)=3$. In view of  $0\to \fm^2 \to \fm\to \fm/\fm^2\to 0$, we observe $$\ell(\fm^2)=\ell(\fm)-3=5-3=2.$$ Either $\mu(\fm^2)=1$ or $\mu(\fm^2)=2$.  Since $R$ is Gorenstein,
$(1,3,2)$ is not admissible. Conclude by this that $\HH(R)=(1,3,1,1)$, i.e., $R$ is stretched. So, $R$ is exact.
Without loss of the generality we can assume that $\ell(R)= 7$. In view of Proposition \ref{3} we may and do assume that $\fm^3\neq 0$. Also, in the light  of Lemma \ref{pro} we may assume that
$\mu(\fm)\leq 4$. There are the following two possibilities:

\begin{enumerate}
\item[i)] One has $\mu(\fm)=4$ (in this case the claim holds without Gorenstein assumption). 	
Indeed, there are $x,y$ and $z$ in $\fm$ such that $xyz\neq0$. Note that
$\ell(R/(xyz))\leq6$. By Nakayama, $xyR/xyzR\neq 0$. Since $0\to xyR/xyzR\to R/xyzR\to R/xyR\to 0$, we have $\ell(R/(xy))<\ell(R/(xyz))\leq6$.
Similarly, $$\ell(R/(x))<\ell(R/(xy))<\ell(R/(xyz))\leq6.$$ Hence, $\ell(R/(x))\leq 4$. Recall that $4\leq\mu(\fm)\leq\ell(R/(x))\leq 4$, i.e., $R$ is exact.

\item[ii)] One has $\mu(\fm)\leq3$. Indeed, by Fact \ref{low}  $\mu(\fm)=3$.
Suppose that $\HH(R):=(1,3,1,1,1)$.  Let us check this case directly. Let $x,y,z,w$ be such that $xyzw\neq 0$ (note that $\fm^4\neq 0$). We deduce from $$\ell(R/(x))<\ell(R/(xy))<\ell(R/(xyz))<\ell(R/(xyzw))\leq6$$
that $\ell(R/(x))\leq 3$. Recall that $3\leq\mu(\fm)\leq\ell(R/(x))\leq 3$. Therefore, $R$ is exact.
By Macaulay's theorem, $(1,3,1,2)$ is not admissible.
The only remaining case is $\HH(R)=(1,3,2,1)$. By definition, $R$ is almost stretched. So, $R$ is exact.
\end{enumerate}
\end{proof}

\begin{remark}
Recall that $(1,3,1,2)$ is not admissible.  Here, we derive it as an application of the theory of exact rings. Suppose it is an $O$-sequence of  a ring $R$. Let $A:=\gr_{\fm}(R)$ and let $x\in A_1$ be such that $x A_2\neq 0$.
We can pick this, because $A_3\neq0$. Note that
$xA=xA_0\bigoplus x A_1\bigoplus  xA_2$.  Clearly, $\dim_k(xA_0)=1$. We see $0\neq xA_1\subset A_2$. Combine this along with  $\dim_k(A_2)=1$
to observe that $\dim_k(xA_1)=1$. Since $0\neq xA_2\subset A_3$ and that $\dim_k(A_2)=1$ we get that $\dim_k(xA_2)=1$.
We apply these to conclude that
$\dim_k(xA)=\sum_{i=0}^2\dim_k(xA_i)=1+1+1=3.$  We look at the exact sequence $0\to xA\to A\to A/xA\to 0$ to deduce
$\ell(A/xA)=\ell(A)-\ell(xA)=7-3=4.$ This proof shows that $x$ is a weak Lefschetz element. By Discussion  \ref{dis} $\ell(A/xA)=\sup\{\mu(A_i)\}=3$. This contradiction says that
$(1,3,1,2)$ is not admissible.
\end{remark}

The  ring presented in \S 2 was of length $8$. Despite of this we have:

\begin{corollary}\label{8}
Let $(R,\fm,k)$ be a Gorenstein artinian local  $k$-algebra, where $k=\overline{k}$ is of zero-characteristic. If $\ell(R)= 8$, then $R$ is exact.
\end{corollary}

\begin{proof} 
By Fact \ref{low} we may assume $\mu(\fm)>2$. In view of Proposition \ref{3} we can assume that $\fm^3\neq 0$. Also, in the light  of Lemma \ref{pro} we may assume that
$\mu(\fm)\leq 5$. Since $R$ is Gorenstein, the final coordinate of its  Hilbert function is one. The following are the possible shape of Hilbert functions:
$$\{(1,6,1),(1,5,1,1),(1,4,2,1),(1,4,1,1,1),(1,3,2,1,1),(1,3,3,1),(1,2,\ldots,1)\}.$$
 The  stretched and  almost stretched cases were handled in Proposition \ref{st} and \ref{als}. Then we may assume that $\HH(R)=(1,3,3,1)$.
We are going to apply the associated graded ring   $\gr_{\fm}(R)$. This is a graded ring with Hilbert function $(1,3,3,1)$.
 Watanabe   proved that if the Hilbert function of a  local ring  is symmetric then the associated graded ring $\gr_{\fm}(-)$ is  Gorenstein (see \cite[Proposition
9]{wat2}).   The assumptions
	imply that $\gr_{\fm}(R)\simeq k[x,y,z]/ I$ where $I$  is Gorenstein, graded and of height three.
	Since  $h_3(k[x,y,z]/ I)=1$, $\gr_{\fm}(R)$ has weak Lefschetz property (see \cite[Corollary 3.7]{Mig}). Combine this along with Discussion \ref{dis}(iii) to observe that $$\rr(\gr_{\fm}(R))=D(\gr_{\fm}(R))=\sup\{h_i(\gr_{\fm}(R))\}=3.$$ We apply   Fact \ref{low}(iv) to deduce $\rr(R)\leq \rr(\gr_{\fm}(R))=3$. Recall that $$3=\mu(\fm)\leq\D(R)\leq\rr(R)\leq 3.$$  So, $\D(R)=\rr(R)$.
The proof is now complete.
\end{proof}

The non-exact ring presented in Example \ref{e9m} is of length 9 and it contains $\mathbb{C}$ (it is not  Gorenstein). Despite of this, there is a complete list of artinian Gorenstein $k$-algebra of length $9$. The possible
Hilbert functions are $$\{  (1,2,\ldots,1),(1,3,2,2,1),(1,3,3,1,1),(1,4,3,1),(1,4,2,1,1),(1,5,2,1),
(1,5,1,1,1),(1,6,1,1),(1,7,1)\}.$$The cases $\fm^2=0$, stretched, and almost stretched are exact. Thus, the only nontrivial
cases are $(1,3,3,1,1)$ and $(1,4,3,1)$.

\begin{fact}\label{93}(See \cite[\S 2]{ca})
Let $(R,\fm,k)$ be a Gorenstein artinian local  $k$-algebra, where $k$ is of zero-characteristic. If $\HH(R)=(1,3,3,1,1)$, then $R:=\frac{k[x_,x_2,x_3]}{I}$ where $I$ is of the form:
\begin{enumerate}
\item[i)]  $(x_1x_2 +x^2_3,x_1x_3,x_2^2 -tx^2_3 -x^{\ell}_1)$
\item[ii)] $ (x_1x_2,x_1x_3,x_2x_3,x^3_2 -x_1^{\ell},x^3_3 -x_1^{\ell})$
\item[iii)] $(x_1x_2,x_2x_3,x^2_1,x_1x^2_3 -x_2^{\ell},x^3_3-x_2^{\ell})$
\item[iv)] $(x_1x_2,x_2x_3,x^2_1 -x^{\ell-1}_3,x_1x^2_3,x^3_2 -x_3^{\ell})$
\item[v)] $(x_1x_2 -x^{\ell-1}_3 ,2x_1x_3 +x^2_2,x^2_1,x_1x^2_3,x_2x^2_3)$
\item[vi)] $(x_1x_2,x^2_1 -x^{\ell-1}_3,x^2_2 -x^{\ell-1}_3,x_1x^2_3,x_2x^2_3 )$.
\end{enumerate}
\end{fact}

\begin{example}
Let $R:=\frac{k[x_,x_2,x_3]}{(x_1x_2,x_2x_3,x^2_1,x_1x^2_3 -x_2^{\ell},x^3_3-x_2^{\ell})}$. Then $R$ is exact.
\end{example}

\begin{proof}
It is easy to see that $\fm^2=(x_2^2,x_3^2,x_1x_3)$. Set $\xi:=x_1+x_2+x_3$. Then
$$\fm\xi=(x_1\xi,x_2\xi,x_3\xi)=(x_1x_3,x_2^2,x_3x_1+x_3^2)=(x_2^2,x_3^2,x_1x_3)=\fm^2.$$
In view of Fact \ref{main}, $R$ is exact.
\end{proof}

\begin{fact}\label{94}(See \cite[\S 5]{ca})
Let $(R,\fm,k)$ be a Gorenstein artinian local  $k$-algebra, where $k$ is of zero-characteristic. If $\HH(R)=(1,4,3,1)$, then $R:=\frac{k[x_,x_2,x_3,x_4]}{I}$ where $I$ is of the form:
\begin{enumerate}
\item[i)]  $ \left(x_1x_2 +x^2_3,x_1x_3,x^2_2 -ax^2_3-x_1^{\ell-1},x_ix_4,x_4^2-x_1^{3}\right)_{i<4}$ for some $a\in k$,
\item[ii)] $\left(x^2_1,x^2_2,x^2_3 +2x_1x_2,x_ix_4,x_4^2 -x_1x_2x_3\right)_{i<4}$
\item[iii)] $\left(x^2_1,x^2_2,x^2_3,x_ix_4,x^2_4 -x_1x_2x_3\right)_{i<4}$
\item[iv)] $\left(x_1x_2,x_1x_3,x_2x_3,x^3_2 -x^3_1,x^3_3 -x^3_1,x_ix_4,x^2_4-x^3_1\right)_{i<4}$
\item[v)] $\left(x^2_1,x_1x_2,x_2x_3,x^3_2 -x^3_3,x_1x^2_3 -x^3_3,x_ix_4,x^2_4 -x^3_3\right)_{i<4}$
\item[vi)] $\left(x^2_1,x_1x_2,2x_1x_3 +x^2_2,x^3_3,x_2x^2_3,x_ix_4,x^2_4 -x_1x^2_3\right)_{i<4}$.
\end{enumerate}
\end{fact}

\begin{example}
Let $R:=\frac{k[x_,x_2,x_3,x_4]}{(x^2_1,x^2_2,x^2_3,x_ix_4,x^2_4 -x_1x_2x_3)}$. Then $R$ is exact provided $\Char (k)\neq 2$.
\end{example}

\begin{proof}
It is easy to see that $\fm^2=(x_1x_2,x_1x_3,x_2x_3)$.
Set $\xi:=x_1+x_2+x_3$. Then $\fm\xi:=(x_1x_2+x_1x_3,x_1x_2+x_2x_3,x_1x_3+x_2x_3)$. Since $\Char (k)\neq 2,$ 
$$x_1x_2=\frac{1}{2}(x_1x_2+x_1x_3)+\frac{1}{2}(x_1x_2+x_2x_3)-\frac{1}{2}(x_1x_3+x_2x_3).$$
Thus, $x_1x_2\in\fm\xi$. Therefore, $\fm\xi=(x_1x_2,x_1x_3,x_2x_3)=\fm^2$.\footnote{If $\Char (k)=2$, then $\fm\xi\subsetneqq\fm^2$. Indeed, $x_1x_2+x_1x_3=(x_1x_2+x_2x_3)+(x_1x_3+x_2x_3)$. Therefore,
$\dim_{k}(\fm\xi)_2<3=\dim_{k}(\fm^2)_2$. From this $\fm\xi\subsetneqq\fm^2$.} In view of Fact \ref{main}, $R$ is exact.
\end{proof}

Here, we study exactness of Gorenstein rings with $\fm^4=0$  in some examples:

\begin{example} \label{ew1}
Let $A:=\mathbb{Q}[u, v, x, y,z]$ and let $I:=\left\{f\in A|f(\partial_u,\partial_v,\partial_x,\partial_y,\partial_z\right)(u^2x + uvy + v^2z)=0\}.$
Then $R:=A/I$ is an  exact Gorenstein ring.
\end{example}

\begin{proof}By a result of Watanabe the Hilbert series
of $R$ is $1+5t+5t^2+1$. Note that $$(2\partial_u\partial_y-\partial_v\partial_z)(u^2x + uvy + v^2z)=0.$$ Also,
$(\partial_u\partial_x-2\partial_v\partial_y)(u^2x + uvy + v^2z)=0$. These mean that $2uy-vz\in I$ and $ux-2vy\in I$.
Let $$J:=\langle z^2, yz, xz ,uz, y^2, xy ,2uy-vz, x^2, vx, ux-2vy, v^3, uv^2, u^2v, u^3\rangle.$$
It is easy to see that $J\subset I$ and that the Hilbert  series
of $A/J$ is $1+5t+5t^2+1$. We look at the exact sequence of graded modules $0\to I/J\to A/J\to R\to 0$. By using the additivity of Hilbert function,
$I/J=0$. In particular, $J= I$ and that $R$ is local.
Let $\fa:=( x, y, z, uv, u^2, v^2)$. Then $(0:_Ru)\subset \fa$ and $\fm \fa=u\fa$. In the light of Fact  \ref{main}, $\mu(\fa)=\ell(R/uR)$.
Therefore, $\D(R)=\rr(R)=6$.
\end{proof}

\begin{lemma}\label{lem:m2}
	Let \(R\) be Artinian, with embedding dimension \(n>0\). If
	\(\xi\in\mm\setminus\mm^2\) and
	\(\mm^2=\xi\mm+\mm^3\), then
	\[
	\mm^2=\xi\mm,\qquad \D(R)=\rr(R)=n.
	\]
\end{lemma}
\begin{proof}
	Apply Nakayama's lemma to \(\mm^2/\xi\mm\). The hypothesis says this
	module equals its product with \(\mm\), so it is zero. Hence
	\(\mm^2\subseteq\xi R\). Modulo \(\mm^2\), the principal ideal
	\(\xi R\) is a one-dimensional subspace of \(\mm/\mm^2\). It follows
	that \(\ell(R/\xi R)=1+n-1=n\). Now use
	\(n=\mu(\mm)\le \D(R)\le \rr(R)\le n\).
\end{proof}

\begin{theorem} \label{thm:nine}
	Let \(R\) be an Artinian Gorenstein local algebra over an algebraically
	closed field of characteristic zero. If \(\ell(R)\le9\), then \(R\) is exact.
\end{theorem}
\begin{proof}
	For length at most eight, use the previous results.
	At length nine, the classification used in the paper leaves only the
	Hilbert functions \((1,3,3,1,1)\) and \((1,4,3,1)\), besides cases
	already covered by those methods. Use the normal forms quoted as
	Facts 3.12 and 3.14, respectively, from Casnati's classification. The following tables give a basis \(q\) of
	\(\mm^2/\mm^3\) and a suitable \(\xi\). Variables in the tables mean
	their classes in the local quotient; indices refer to those facts.
	
	\begin{center}\small
		\begin{tabular}{@{}lll@{}}
			\toprule
			Fact 3.12 & Basis \(q\) of \(\mm^2/\mm^3\) & \(\xi\)\\\midrule
			(i) & \((x_1^2,x_2x_3,x_3^2)\) & \(x_1+x_3\)\\
			(ii) & \((x_1^2,x_2^2,x_3^2)\) & \(x_1+x_2+x_3\)\\
			(iii), (iv) & \((x_1x_3,x_2^2,x_3^2)\) & \(x_2+x_3\)\\
			(v), (vi) & \((x_1x_3,x_2x_3,x_3^2)\) & \(x_3\)\\\bottomrule
		\end{tabular}
	\end{center}
	
	For (i), the products \(x_1\xi,x_2\xi,x_3\xi\), modulo \(\mm^3\),
	are \(q_1,q_2-q_3,q_3\), hence independent. For (ii) the three products
	are the three displayed squares. For (iii) and (iv) they are
	\(q_1,q_2,q_3\). The same holds for (v) and (vi) with \(\xi=x_3\).
	Possible higher terms in the displayed relations do not affect these
	computations. Thus multiplication
	\(\mm/\mm^2\xrightarrow{\,\xi\,}\mm^2/\mm^3\) has rank three
	in each case. Lemma~\ref{lem:m2} gives \(\D=\rr=3\).
	
	\begin{center}\small
		\begin{tabular}{@{}lll@{}}
			\toprule
			Fact 3.14 & Basis \(q\) of \(\mm^2/\mm^3\) & \(\xi\)\\\midrule
			(i) & \((x_1^2,x_2x_3,x_3^2)\) & \(x_1+x_3\)\\
			(ii) & \((x_1x_2,x_1x_3,x_2x_3)\) & \(x_3\)\\
			(iii) & \((x_1x_2,x_1x_3,x_2x_3)\) & \(x_1+x_2+x_3\)\\
			(iv) & \((x_1^2,x_2^2,x_3^2)\) & \(x_1+x_2+x_3\)\\
			(v) & \((x_1x_3,x_2^2,x_3^2)\) & \(x_2+x_3\)\\
			(vi) & \((x_1x_3,x_2x_3,x_3^2)\) & \(x_3\)\\
			\bottomrule
		\end{tabular}
	\end{center}
	
	The product by \(x_4\) is zero modulo \(\mm^3\) in all these normal
	forms. For (i) the first three products are again
	\(q_1,q_2-q_3,q_3\). For (ii) they are \(q_2,q_3,-2q_1\).
	For (iii), the multiplication matrix on the first three variables is
	\[
	\begin{pmatrix}1&1&0\\1&0&1\\0&1&1\end{pmatrix},
	\qquad\det=-2.
	\]
	For (iv), (v) and (vi), the products are the respective basis vectors.
	All maps therefore have rank three in characteristic zero.
	Lemma~\ref{lem:m2} gives \(\D=\rr=4\) in every case.
\end{proof}
\begin{example} \label{ew}
Let $A:=k[X_1, \ldots, X_5]$ be the polynomial ring over any field $k$ of zero-characteristic.
Let $I$ be the ideal generated by the
following quadric relations:
$$\Delta:=\{X_1X_3 + X_2X_3, X_1X_4 + X_2X_4, X_3^2 + X_1X_5 - X_2X_5,
X^2_4 + X_1X_5 - X_2X_5, X^2_1, X_2^2
, X_3X_4, X_3X_5, X_4X_5, X_5^2\}.$$
Let $R := A/I$. The ring $R$ is   Gorenstein and $\fm^4=0$. The following holds:\begin{enumerate}
\item[i)] One has $\ell(R/\sum_{i=1}^5 x_i)\neq\inf\{\ell(R/rR)\}$.
\item[ii)] The ring $R$ is exact. In fact, there are infinitely many
$r$ such that $\ell(R/rR)=\rr(R)=\D(R)=5$.
\end{enumerate}
\end{example}

\begin{proof}
In view of \cite{ga}, $R$ is a local Gorenstein ring with Hilbert series $\HH(t) = 1+5t+5t^
2+t^3$. Hence $\fm^4=0$.
Also, as a $k$-vector space, it has a basis consisting of $$\Gamma:=\{
1, x_1, x_2, x_3, x_4, x_5, x_1x_2, x_1x_3, x_1x_4, x_1x_5, x_2x_5, x_1x_2x_5 \}.$$The relations
$x_3^2 + x_1x_5 - x_2x_5=
x^2_4 + x_1x_5 - x_2x_5=0$ implies that $x^2_3=x_4^2$ and $x^2_4 + x_1x_5 =x_2x_5$.
Let $$\xi:=x_1+Ax_2+Bx_3+Cx_4+Dx_5\in R_1$$ be an element.
We are going to force conditions for which $\fm\xi=\fm^2$.
Let us compute $x_i\xi:$
\begin{enumerate}
\item[1)] $x_1\xi=Ax_1x_2+Bx_1x_3+Cx_1x_4+Dx_1x_5$,
\item[2)] $x_2\xi =x_1x_2+Bx_2x_3+Cx_2x_4+Dx_2x_5=x_1x_2-Bx_1x_3-Cx_1x_4+Dx_1x_5+Dx_3^2$,
\item[3)] $x_3\xi=x_3x_1+Ax_2x_3+Bx_3^2=(1-A)x_3x_1+Bx_3^2$,
\item[4)]  $x_4\xi=x_4x_1+Ax_2x_4+Cx_4^2=(1-A)x_1x_4+Cx_3^2$,	
\item[5)] $x_5\xi=x_5x_1+Ax_2x_5=(A+1)x_1x_5+Ax_3^2$.
\end{enumerate}
 We now  are ready to prove the desired claims:

\begin{enumerate}
\item[i)] Let $A=B=C=D=1$. We have $\fm\sum_{i=1}^5 x_i\neq \fm^2$, because $\dim_k(\fm\sum_{i=1}^5 x_i)_2<5=\dim_k(\fm^2)_2$.

\item[ii)]  Let $A:=-1$. We claim that $\fm\xi=\fm^2$ for any nonzero $B$, $C$ and $D$. Indeed,
from $5)$ we get that $x_3^2\in \fm\xi$. In view of $4)$ and  $3)$ we see $\{x_3x_1,x_1x_4\}\subset\fm\xi$.
Look at
\[\begin{array}{ll}
x_1\xi+x_2\xi
=(A+1)x_1x_2+(2D)x_1x_5+Dx_3^2=(2D)x_1x_5+Dx_3^2.
\end{array}\]Conclude by this that  $x_1x_5\in\fm\xi$.
In the light of  $x_3^2 + x_1x_5 - x_2x_5=0$ we see $x_2x_5\in\fm\xi$. Since
$\fm^2=\langle x_1x_2,x_3x_1,x_1x_4,x_2x_5,x_1x_5\rangle$ we have
$\fm\xi=\fm^2$. This says that $\xi$ is a Lefschetz element. The Hilbert function is unimodal. Combine these along with  Discussion \ref{dis}(iii) to conclude
that $\rr(R)=5$. Recall that $5\leq\D(R)\leq\rr(R)=5$. \end{enumerate}This is what we want to prove.
\end{proof}

\section{Higher dimensional exactness }

Suppose $R$ is a homomorphic image of a Gorenstein ring $A$ of the same dimension as $R$ (e.g., $R$ is complete). Let $M$ be a finitely generated $R$-module of dimension $d$. If $d = 0$, we set the \textit{Dilworth number} of $M$ as $\D_h(M) := \D(M)$. 
Note that, for any $i$, we have:
\(
\dim(\Ext_A^{d-i}(M, A)) \leq i.
\)
Inductively, the \textit{homological Dilworth number} of $M$ is defined as:
\[
\D_h(M) := \sum_{i=0}^{d-1} \binom{d-1}{i} \D_h(\Ext_A^{d-i}(M, A)) + \e(M),
\]
where $\e(M)$ denotes some specific term depending on the module $M$. This definition is independent of the choice of $A$ (see \cite[Proof of 3.9(a)]{vas}).

\begin{fact}\label{G}(Gunston \cite{dil})
Adopt the above notation. If $\dim R=1$, then $\D(R)=\D_h(R)$.
\end{fact}

\begin{lemma}\label{comp}
Adopt the above notation. Then
$\D_h(M) =\D_h(M\otimes_R\widehat{R})$.
\end{lemma}

\begin{proof}
	The proof is by induction on $d:=\dim M$. Since   finite-length module are complete, we may assume $d>0$. This is well-known that $\e(M)=\e(\hat{M})$. Recall that $R$ is  homomorphic image of a Gorenstein ring $(A,\fm_A)$ of same dimension  as of $R$.
Let $i<d$. By induction,  $\D_h(\Ext^{d-i}_A(M,A)\otimes_R\widehat{R})=\D_h(\Ext^{d-i}_A(M,A))$. Note that $\fm_A$-adic completion of a finitely generated $R$-module
$N$ is the same as of its $\fm_R$-adic completion,
when we view $N$ as an $A$-module via the map $A\to R$.  Then, $\Ext^{d-i}_A(M,A)\otimes_R\widehat{R}\simeq\Ext^{d-i}_{\hat{A}}(\hat{M},\hat{A})$.  Therefore, $\D_h(\Ext^{d-i}_A(M,A))=\D_h(\Ext^{d-i}_{\hat{A}}(\hat{M},\hat{A}))$. By definition, $\D_h(M)=\D_h(\hat{M})$.
\end{proof}
If $R$ is not  homomorphic image of a Gorenstein ring,  define
$\D_h(M) :=\D_h(M\otimes_R\widehat{R})$.  By the above lemma, this is well-define.
By a \textit{formal} module we mean a finitely generated $\widehat{R}$-module.
Also,  a formal module $\mathcal{M}$ is  called \textit{algebraic} if there is a finitely
generated $R$-module $M$ such that $\widehat{M}\simeq \mathcal{M}$.

\begin{fact}\label{algart}
i) Any finite length $\widehat{R}$-module  $\mathcal{M}$ is algebraic.

ii) Let $M$ be an $R$-module such that $\ell_R(\widehat{M})<\infty$. Then  $\ell_R(M)=\ell_{\widehat{R}}(\widehat{M})$.
\end{fact}

\begin{proof}
This follows by an easy induction and we left the routine details to the reader.
\end{proof}

\begin{corollary}\label{mc}
Let  $\mathcal{I}\vartriangleleft \widehat{R}$ be a full parameter ideal in a local ring $\widehat{R}$.
Then there is a full parameter ideal $I$ in $R$ such that $\ell(R/I)=\ell(\widehat{R}/\mathcal{I})$.
In particular, $\m(R)=\m(\widehat{R})$.
\end{corollary}

\begin{proof}
	By Fact \ref{algart}(i), there is a finitely generated $R$-module $M$ such that
$\widehat{M}\simeq\widehat{R}/\mathcal{I}$. We claim that $M$ is cyclic. Indeed, we look at the following exact sequence of formal modules $0\to \mathcal{I}\to \widehat{R}\stackrel{\mathcal{F}}\to \widehat{R}/\mathcal{I} \to 0.$
Since $\Hom_{\widehat{R}}(\widehat{R},\widehat{R}/\mathcal{I})\simeq\Hom_{R}(R,M)\otimes \widehat{R}$ there is an  $F\in \Hom_{R}(R,M)$
such that $\widehat{F}=\mathcal{F}$.
 Due to faithfully flatness of completion, $\coker(F)=0$, i.e., $F$ is surjective. So $M$ is cyclic. Put $I:=\ker F\lhd R$. There is a natural map $I\otimes \widehat{R}\to \mathcal{I}$.
We apply 5-lemma to observe that   $\widehat{I}\simeq\mathcal{I}$.
By Fact \ref{algart}(ii),  $\ell(R/I)=\ell(\widehat{R}/\mathcal{I})$. Now, we show $I$ is full parameter:
$$\mu(I)=\dim _{\frac{R}{\fm}}(\frac{I}{\fm I})=\dim ((\frac{I}{\fm I})\otimes \widehat{R})=\dim _{\frac{\widehat{R}}{\widehat{\fm}}}  (\frac{\mathcal{I}}{\widehat{\fm} \mathcal{I}})=\mu(\mathcal{I}).$$This implies that $\m(R)\leq\m(\widehat{R})$. Let $I\lhd R$ be a parameter ideal. Then $\widehat{I}\lhd \widehat{R}$ is a parameter ideal. Conclude from this that $\m(R)\geq\m(\widehat{R})$. So, $\m(R)=\m(\widehat{R})$.
\end{proof}

\begin{definition}Let $R$ be any local ring.
We say $R$ is homologically exact if $\D_h(R)=\m(R)$.
\end{definition}

The following may answer Question \ref{mat2}(ii):

\begin{proposition}
Let $R$ be a 1-dimensional local ring. Then $R$ is exact if and only if $R$ is homologically exact.
\end{proposition}

\begin{proof}Note that $\ell(R/xR)$ is finite if and only if $x$ is parameter. From this, $\m(R)=\rr(R)$.
By Corollary \ref{mc}, $\m(R)=\m(\widehat{R})$. It follows from Lemma \ref{comp} that  $\D_h(R)=\D_h(\widehat{R})$. Hence, without loss of the generality, we may assume that $R$ is complete.
By Fact \ref{G}
$\D(R)=\D_h(R)$.
It is now clear  that  $R$ is exact if and only if $R$ is homologically exact.
\end{proof}

\begin{lemma}\label{ab}
 Let $(R,\fm,k)$ be a local ring and  $M$  a module such that $\fm M=0$. Then $M$ is exact.
\end{lemma}

\begin{proof}
The condition $\fm M=0$ implies that  $M/xM=M$ for any $x\in\fm$. Hence, $\ell(M)=\rr(M).$
Also, $\fm M=0$ implies that
any
submodule $N$ of $M$ is  a vector space.
It turns out that $\dim_{k}N=\ell(N)=\ell(N/\fm N)=\mu(N)$.
Thus, $\D(M)=\dim_{k}M$. So,
$\D(M)=\dim_{k}M=\ell(M)=\rr(M).$
\end{proof}

\begin{lemma}\label{III}
Let $(R,\fm)$ be a generalized Cohen-Macaulay  ring of  dimension $d>0$ and that $k$ is an infinite. The following assertions hold: \begin{enumerate}
\item[i)] One has $\m(R)\leq\e(R)+\sum_{i=0}^{d-1}(^{d-1}_i)\ell(\HH^i_{\fm}(R)).$
\item[ii)]  If  $R$ is Buchsbaum, then $\m(R)=\e(R)+\sum_{i=0}^{d-1}(^{d-1}_i)\ell(\HH^i_{\fm}(R)).$
\end{enumerate}
\end{lemma}

\begin{proof}
i) For each parameter ideal $\fq$, set $I(\fq) := \ell(R/ \fq ) -\e(\fq,R)$. Since $R$ is generalized Cohen-Macaulay,
$I(R):=\sup\{I(\fq) \}$ is finite. In fact, $I(R)=\sum_{i=0}^{d-1}(^{d-1}_i)\ell(\HH^i_{\fm}(R)).$
In view of \cite[Theorem 14.14]{Mat}, $\fm$ has a reduction $\fq_0$ which is   a parameter ideal.
Due to \cite[Theorem 14.13]{Mat} $\e(R)=\e(\fq_0,R)$. Putting all of these together, we have
\[\begin{array}{ll}
\m(R)&\leq\ell(R/ \fq_0 )\\
&=I(\fq_0)+\e(\fq_0,R)\\
&=I(\fq_0)+\e(R)\\
&\leq\sup\{I(\fq) \}+\e(R)\\
&=\sum_{i=0}^{d-1}(^{d-1}_i)\ell(\HH^i_{\fm}(R))+\e(R).
\end{array}\]

ii)   Let $\fq_0$ be a parameter sequence such that $\m(R)=\ell(R/ \fq_0 )$.
 The natural surjection $R/\fq_0^n\to R/\fm^n$ shows that $\ell(R/\fq_0^n)\geq \ell(R/\fm^n)$.
Consequently, $\e(R)\leq\e(\fq_0,R)$. Also, $I(\fq_0)=\sup\{I(\fq):\fq\textit{ is parameter}\}$, because over Buchsbaum rings $I(\fq)$ is independent of $\fq$.  Putting these together,
\[\begin{array}{ll}\m(R)&=\ell(R/ \fq_0 )\\
&=I(\fq_0)+\e(\fq_0,R)\\
&\geq I(\fq_0)+\e(R)\\
&=\sup\{I(\fq) \}+\e(R)\\
&=\sum_{i=0}^{d-1}(^{d-1}_i)\ell(\HH^i_{\fm}(R))+\e(R)\\
&\stackrel{(i)}\geq\m(R).
\end{array}\]
Thus, $\m(R)=\sum_{i=0}^{d-1}(^{d-1}_i)\ell(\HH^i_{\fm}(R))+\e(R)$.
\end{proof}

The following is a higher-dimensional version of  \cite[Theorem 2.2]{ii}:

\begin{proposition}\label{4.9}
Any  Buchsbaum ring of positive dimension with infinite residue field is homologically exact.
\end{proposition}

\begin{proof}
In view of \cite[Lemma 1.13]{bus}  a ring is Buchsbaum if and only its $\fm$-adic completion is Buchsbaum. By Corollary \ref{mc} $\m(R)=\m(\widehat{R})$. We may assume that $R$ is complete. Let $A$ be a Gorenstein ring of same dimension as of $R$ that maps onto $R$.
By  Lemma \ref{III}
$\m(R)=\sum_{i=0}^{d-1}(^{d-1}_i)\ell(\HH^i_{\fm}(R))+\e(R).$ Since $R$ is Buchsbaum $\fm\HH^i_{\fm}(R)=0$ for al $i<\dim R$.
In view of Lemma \ref{ab} $\ell(\HH^i_{\fm}(R))=\mu(\HH^i_{\fm}(R))=\D(\HH^i_{\fm}(R))$. Since  $\dim (\HH^i_{\fm}(R))=0$, we have $\D(\HH^i_{\fm}(R))=\D_h(\HH^i_{\fm}(R))$.
In view of \cite[Corollary 4.2.3]{dil} Dilworth number of a finite length module is stable under taking Matlis duality.   In view of independence theorem and local duality from
Grothendieck's local cohomology modules, we have:
 \[\begin{array}{ll}
\m(R)&=\sum_{i=0}^{d-1}(^{d-1}_i)\ell(\HH^i_{\fm}(R))+\e(R)\\
&=\sum_{i=0}^{d-1}(^{d-1}_i)\D_h(\HH^i_{\fm}(R))+\e(R)\\
&=\sum_{i=0}^{d-1}(^{d-1}_i)\D_h(\HH^i_{\fm_A}(R))+\e(R)\\
&=\sum_{i=0}^{d-1}(^{d-1}_i)\D_h(\HH^i_{\fm_A}(R)^v)+\e(R)\\
&=\sum_{i=0}^{d-1}(^{d-1}_i)\D_h(\Ext^{d-i}_A(R,A))+\e(R).
\end{array}\]The last one is
$\D_h(R)$. This completes the proof.
\end{proof}
An example of a homologically exact ring that is not Buchsbaum is given below:
\begin{example}
Look at $R := \frac{\mathbb{Q}[[x, y]]}{x\fm^2}$. This is well known that
$\e(R)=\sup\{\mu(\fm^n):n\gg 0\}$. The set $\{x^iy^j:0<i+j=n\}$ is a generator for $\fm^n$. Due to $x^3=x^2y=xy^2=0$ we observe that $\fm^n=y^nR$ for all $n\geq3$. From this,
$\e(R)=1$. Also, $$
R/yR\simeq\frac{\mathbb{Q}[[x,y]]/(x^3,x^2y,xy^2)}{(y,x^3,x^2y,xy^2)/(x^3,x^2y,xy^2)}
\simeq \mathbb{Q}[[x,y]]/(y,x^3,x^2y,xy^2)
\simeq\mathbb{Q}[[x,y]]/(y,x^3).$$
As a vector-space, the last one is $\mathbb{Q}\oplus \mathbb{Q}x\oplus \mathbb{Q}x^2$. Thus, $\m(R)\leq\ell(R/yR)\leq3$. Note that $\HH^0_{\fm}(R)=xR$. This is principal. But its submodule $\fm\HH^0_{\fm}(R)=(x^2,xy)$ is not principal.
Consequently, $\D(\HH^0_{\fm}(R))>1$. Recall from \cite[Proposition 2.4]{ii} that $\D(R)=\e(R)+\D(\HH^0_{\fm}(R))$ and $\D(R)\leq\m(R)$. Putting these together: $3\leq\D(R)\leq\m(R)\leq3$.
Hence $\D(R)=\m(R)$. Therefore, $R$ is homologically exact. But $R$ is not Buchsbaum, because $\fm\HH^0_{\fm}(R)\neq 0$.
\end{example}
Rings of Dilworth number 2 are so special:

\begin{observation}(After Bass-Matlis)
Let $R$ be  a 1-dimensional reduced  ring  and
$R'$ be its normalization such that $\mu_R(R')<\infty$.  The following are equivalent:
a) $\D(R)\leq2$,
b) Any ring between $R$ and $R'$ is Gorenstein,
c) Any ring between $R$ and $R'$ is complete-intersection.
\end{observation}

\begin{proof}
The equivalence $a)\Leftrightarrow b)$ is in \cite[\S 7]{bass}.
Assume b). Let $A$ be any ring  between $R$ and $R'$.  Any ring between $A$ and $A'=R'$ is Gorenstein. By $b)\Leftrightarrow a)$ we have $\D(A)=2$.  Let $\fm$ be any maximal ideal of $A$. Then $\mu(\fm A_{\fm})\leq 2$. Recall that reduced rings satisfy in the Serre's condition $R(0)$. This, for the possibility $\mu(\fm A_{\fm})\leq 1$, implies  that $A_{\fm}$ is regular. Without loss of the generality  we may assume that $\mu(\fm A_{\fm})=2$.
Note $A_{\fm}$ is reduced. Thus, $\dim A_{\fm}\neq 0$. In particular,   $A_{\fm}$ is  Cohen-Macaulay and that $\emb(A_{\fm})=\dim A_{\fm}+1$. In view of \cite[Ex. 21.2]{Mat}, $A_{\fm}$ is complete-intersection. In any cases $c)$ follows.
Clearly, $c)$ implies b).
\end{proof}

The reduced assumption is important:
\begin{example}
Look at $R := \frac{\mathbb{Q}[[x, y]]}{(x^2, xy)}$. Then $\depth (R)=0$ and $\dim R=1$.
Note that  $\e(R)=\lim_{n\to\infty}
\frac{\ell(R/ \fm^n)(d)^!}{n^{d}}=1.$
 One has $\rad(y)=(x,y)$. Thus $y$ is a parameter element. Also,$$ R/yR=\frac{\mathbb{Q}[[x,y]]/(x^2,xy)}{(x^2,y)/(x^2,xy)}\simeq \mathbb{Q}[[x,y]]/(x^2,y)\simeq\mathbb{Q}\oplus \mathbb{Q}x,$$as vector-spaces. This implies that $\ell(R/yR)\leq2$. Hence, $1\leq\m(R)\leq2$. However, $\m(R)\neq1$ because $R$  is not regular. Therefore, $\m(R)=2.$
 On the other hand $\D(R)\leq\m(R)=2$ and that $\mu(\fm)=2$.  So, $\D(R)=2$. But, $R$ is not even Cohen-Macaulay. So, the implication $a)\Rightarrow b)$
 in the above observation is not true without the reduced assumption.
\end{example}

Recall that a system of parameters for a non-zero finitely generated
$R$-module $M$ is a sequence $\underline{x}:=x_1,\ldots,x_n$  where $n:=\dim M$   such that
$M/\underline{x}M$ is artinian. By $\m(M)$ we mean $\inf\{\ell(M/\underline{x}M):\underline{x}\textit{ is a  M-parameter sequence} \}$ and we call it  \textit{Matsumura number} of $M$.
Also, $M$ is called homologically exact if $\m(M)=\D_h(M)$.

\begin{lemma}\label{4.13}
Let $R$ be a local ring and $M$ an module of dimension same as of $\dim R>0$. Then
$M$ is maximal Cohen-Macaulay if and only if  $\e(M)=\D_h(M)$.
\end{lemma}

\begin{proof}
We may assume $R$ is complete. Let $d:=\dim M=\dim R$. Suppose $R$ is  homomorphic image of a Gorenstein ring $A$ of same dimension  as of $R$.
 First, assume that $\e(M)=\D_h(M)$.
 Then $\Ext^{d-i}_A(M,A) =0$ for all $i<d$.
By local duality, $\HH^i_{\fm_A} (M)=0$ for all $i<d$. By the independence theorem, $\depth_R(M)=d$. So, $M$ is maximal Cohen-Macaulay.
Conversely, assume that $M$ is maximal Cohen-Macaulay. Let $\underline{x}$ be a parameter sequence for $R$. Let $\pi:A\to R$ be the surjection. There are $y_i$ such that $\pi(y_i)=x_i$.
Note that $\rad(y_i)=\fm_A$. Hence, $\underline{y}:=\{y_i\}$ is a parameter sequence for $A$, because $\dim A=\dim R$. We view  $M$ as an $A$-module via $\pi$. By this, $\underline{y}$ is $M$-sequence, i.e.,
$M$ is maximal Cohen-Macaulay as an $A$-module. Then $\Ext^i_A(M,A)=0$ for all $i>0$. By definition,  $\e(M)=\D_h(M)$.
\end{proof}

The following is a higher-dimensional version of \cite[Theorem 3.3]{ii}:

\begin{proposition}\label{4.14}
Let $(R,\fm,k)$ be a local Cohen-Macaulay ring with infinite residue field, positive dimension and with a canonical module $\omega_R$. Then
$$\m(\omega_R)=\D_h(\omega_R)=\e(\omega_R)=\m(R)=\D_h(R)=\e(R).$$
\end{proposition}

\begin{proof}
Since $\omega_R$ is maximal Cohen-Macaulay and by  the above lemma,
$\e(\omega_R)=\D_h(\omega_R)$.
Let $\fq$ be any $\fm$-primary ideal. Recall that the  natural surjection $\frac{\omega_R}{\fq^n\omega_R}\to \frac{\omega_R}{\fm^n\omega_R}$ shows that $\e(\fq,\omega_R)\geq \e(\omega_R)\ \ (\ast)$.
 Let $\underline{x}$ be a system of parameters.   By a theorem of Serre, $\Sigma_{i>0}(-1)^{i+1}\ell(\HH^i(\underline{x},\omega_R))\geq 0$, i.e., $\e(\underline{x},\omega_R)\leq \ell(\omega_R/\underline{x}\omega_R)$. Let $\underline{x}_0$ be such that the $\inf$ in $\m(\omega_R)$ happens.
Thus, $$\m(\omega_R)=\ell(\omega_R/\underline{x}_0\omega_R)\geq\e(\underline{x}_0,\omega_R)\stackrel{(\ast)}\geq\e(\omega_R).$$Since $k$ is infinite,
 there is a reduction of $\fm$ by a parameter sequence $\underline{x}$ (see \cite[Theorem 14.14]{Mat}).
In view of \cite[Theorem 14.13]{Mat}, $\e(\omega_R)=\e(\underline{x},\omega_R)$. Since $\underline{x}$ is $\omega_R$-sequence,
$\Sigma_{i>0}(-1)^i\ell(\HH^i(\underline{x},\omega_R))= 0$, i.e.,
 $\e(\underline{x},\omega_R)=\ell(\omega_R/\underline{x}\omega_R)$. Hence  $$\e(\omega_R)=\e(\underline{x},\omega_R)=\ell(\omega_R/\underline{x}\omega_R)\geq\m(\omega_R).$$
Therefore,  $\m(\omega_R)=\D_h(\omega_R)=\e(\omega_R)$.
Similarly, $\m(R)=\D_h(R)=\e(R)$. If $\omega_R$ has a rank then its rank is one and so $\e(\omega_R)=\e(R)$. But $\omega_R$ has a rank if and only if $R$ is generically Gorenstein.
Let us handle the general case by the following trick of Ikeda. Let $\underline{x}$ be any parameter sequence.  Then  $\underline{x}$ is  both $R$-regular and $\omega_R$-regular. Also, $\frac{\omega_R}{\underline{x}\omega_R}\simeq\omega_{R/\underline{x}R}$. Recall that $\overline{R}:=R/\underline{x}R$ is zero-dimensional. In particular,  it is complete and the only  indecomposable injective module is $E_{\overline{R}}(k)$. Thus, $\omega_{\overline{R}}\simeq E_{\overline{R}}(k)$. In view of Matlis duality $$\ell(\overline{R})=\ell(\overline{R}^v)=\ell(E_{\overline{R}}(k))=\ell(\omega_{\overline{R}})=\ell(\frac{\omega_R}{\underline{x}\omega_R})\quad(+)$$

Let $\underline{y}=y_1,\ldots,y_d\subset\fm$ be any sequence of elements. If
 $\underline{y}$
is an $R$-parameter sequence,  then $\ell(R/\underline{y}R)<\infty$. From this $\ell(R/\underline{y}R\otimes\omega_R)<\infty$. Since for any
strict subset $\underline{y}'$ of  $\underline{y}$ we have $\ell(R/\underline{y}'R\otimes\omega_R)=\infty$, it is clear that $\underline{y}$
is an $\omega_R$-parameter sequence. Conversely, suppose $\underline{y}$
is an $\omega_R$-parameter sequence. Then $\underline{y}$ is an $\omega_R$-regular sequence. (If $\pd(\omega_R)<\infty$, then by  Auslander zero-divisor, $\underline{y}$ is an $R$-regular sequence. But, $\pd(\omega_R)<\infty$ if and only if $R$ is Gorenstein. Instead, we use the following trick:) Recall that
 $\frac{\omega_R}{(y_1,\ldots,y_i)\omega_R}$ is Cohen-Macaulay for all $i$.
For any $\fp\in\Ass(\frac{\omega_R}{(y_1,\ldots,y_i)\omega_R})$ we have $\dim R/ \fp=d-i$.
Note that canonical module has  support same as of the prime spectrum. Also, associated prime
ideals of a Cohen-Macaulay module is the minimal elements of its supports. From these,  $$\Ass(\frac{\omega_R}{(y_1,\ldots,y_i)\omega_R})=\Ass(\frac{R}{(y_1,\ldots,y_i)}).$$
This yields that $\dim (\frac{R}{(y_1,\ldots,y_i)})=d-i$ for all $i$.
It turns out that  $\underline{y}$
is an $R$-parameter sequence. 

Combine the previous  paragraph along with $(+)$
to conclude that $\m(\omega_R)=\m(R)$. Consequently: $$\m(\omega_R)=\D_h(\omega_R)=\e(\omega_R)=\m(R)=\D_h(R)=\e(R).$$
The proof is now complete.\end{proof}

\begin{remark}Proposition \ref{4.9}, Lemma \ref{4.13}  and Proposition \ref{4.14}.
All three statements, the nonzero dimension  assumption is needed.   Let
\(
A=k[\varepsilon]/(\varepsilon^2), |k|=\infty.
\)
This is an Artinian Gorenstein, hence Cohen--Macaulay and Buchsbaum,
ring. With the paper's definitions,
\[
\D(A)=\rr(A)=\D_h(A)=1,\qquad \m(A)=\e(A)=\ell(A)=2.
\]
Here the only system of parameters is the empty sequence. Thus \(A\) is
not homologically exact, contrary to Proposition 4.9 as stated. It is
Cohen--Macaulay but \(\e(A)\ne\D_h(A)\), contrary to Lemma \ref{4.13}. Its canonical
module is \(A\), so Proposition \ref{4.14} also fails.\end{remark}
\begin{remark}
i) Here we present a (new) proof of Serre's positivity theorem: $\Sigma_{i>0}(-1)^{i+1}\ell(\HH^i(\underline{x},R))\geq 0$.
Indeed,
recall that $I:=(x_1,\ldots,x_d)$ is a full parameter ideal. We may assume that $d>0$. The assignment $X_i\mapsto x_i + (\underline{x})^2$ gives the surjection
$P:= \frac{R}{I}[X_1, \ldots ,X_d]\twoheadrightarrow \gr_{I}(R) =: G$.
 This yields that
$$\ell(R/I)
= \lim_{n\to\infty} \frac{\ell(P_n) d ^!}{n^{d}}
\stackrel{(\ast)}\geq \lim_{n\to\infty}
\frac{\ell(G_n) d  ^!}{n^{d}}=\lim_{n\to\infty}
\frac{\ell(R/I^n) d  ^!}{n^{d}}=\e(I,R).$$
Since $\e(I,R)=\Sigma_{i\geq0}(-1)^i\ell(\HH^i(\underline{x},R))$, we see $\Sigma_{i>0}(-1)^i\ell(\HH^i(\underline{x},R))\leq 0$. So, the claim follows.

ii) Recall that $(\ast)$ in part i) is the equality if and only if $\underline{x}$ is a  regular sequence. From this we can reprove the following result of Shoutens:
Let $(R,\fm)$ be  a local ring  such that $|R/\fm|=\infty$. Then $R$ is Cohen-Macaulay
if and only if $\m(R)=\e(R)$.
\end{remark}

\section{Concluding remarks and questions  }

In this section we present concluding remarks and questions. To start, we ask with the following.
Are zero-dimensional complete intersections in characteristic zero exact?

\begin{fact}[Watanabe et al.]
	Let $(R, \fm, k)$ be an Artinian equi-characteristic local ring.
	\begin{enumerate}[i)]
		\item Any homogeneous complete intersection of characteristic zero and embedding dimension $\emb(R) = 3$ is exact.
		
		\item Let $f_1, \ldots, f_n \in \mathbb{Q}[X_1, \ldots, X_n]$ be a homogeneous regular sequence of degrees $d_1 \leq \ldots \leq d_n$, and suppose that $d_n \geq d_1 + \ldots + d_{n-1} - n$. Then the ring
		\(
		R := \frac{\mathbb{Q}[X_1, \ldots, X_n]}{(f_1, \ldots, f_n)}
		\)
		is exact.
		
		\item If $d_n \geq d_1 + \ldots + d_{n-1} - n + 2$, the same conclusion as in (ii) holds over any infinite field.
		
		\item If $f_1, \ldots, f_n \in \mathbb{Q}[X_1, \ldots, X_n]$ form a monomial complete intersection, then
		\( 
		R := \frac{\mathbb{Q}[X_1, \ldots, X_n]}{(f_1, \ldots, f_n)}
		\)
		is exact.
	\end{enumerate}
\end{fact}

\begin{proof}
	(i) The ring $R$ is of the form $\frac{K[x,y,z]}{(f_1,f_2,f_3)}$ where the $f_i$ are homogeneous of degrees $d_i$, with $2 \leq d_1 \leq d_2 \leq d_3$. By \cite[Theorem 3.48]{wat3}, we have $\rr(R) = \sup\{\mu(\fm^i)\}$. Hence, $R$ is exact.
	
	(ii) By induction, we may assume $2 \leq d_1$. Then, by \cite[Corollary 3.54]{wat3}, we again have $\rr(R) = \sup\{\mu(\fm^i)\}$.
	
	(iii) This follows from \cite[Remark 3.55]{wat3}.
	
	(iv) Monomial complete intersections over fields of characteristic zero have the Lefschetz property. By \cite[Proposition 3.6]{wat3}, we conclude $\rr(R) = \sup\{\mu(\fm^i)\}$, and thus $R$ is exact.
\end{proof}

It may happen that $\D(R) \neq \sup\{\mu(\fm^n)\}$ even though $R$ is exact?this can occur even for Gorenstein rings. The first such example was constructed by Ikeda \cite[Example 4.4]{i}. 
Ikeda also conjectured that $\D(R) = \sup\{\mu(\fm^n)\}$ for any Artinian Gorenstein local ring of embedding dimension $3$. However, there are counterexamples in positive characteristic; see \cite{bcan}. 

\begin{question}
Let $(R, \fm)$ be a zero-dimensional Gorenstein local ring of embedding dimension $3$ over a field of characteristic zero. Is it true that $\D(R) = \sup\{\mu(\fm^n):=\sup\{\mu(\fm^n)\}$?
\end{question}For a positive answer see {Observation} \ref{prop:wlp}.
\begin{remark}
	\begin{enumerate}[i)]
		\item Let $\mathbb{Q} \subset R$ be an Artinian Gorenstein local ring such that $\mu(\fm) = \mu(\fm^2) = 3$ and $\fm^4 = 0$. Then $R$ is exact; in fact, $\D(R) = \sup\{\mu(\fm^n)\}$. We may assume that $\fm^3 \neq 0$, and in particular, $\mu(\fm^3) = 1$. Hence, $\max\{\mu(\fm^i)\} = 3$. By the proof of Corollary~\ref{8}, we have $\D(R) = \rr(R) = 3$.
		
		\item The assumption $\mathbb{Q} \subset R$ in part (i) is essential. For instance, let
		\(
		R := \frac{\mathbb{F}_2[X,Y,Z]}{(X^2, Y^2, Z^2)}.
		\)
		Clearly, $R$ is an Artinian Gorenstein local ring with $\mu(\fm) = \mu(\fm^2) = 3$ and $\fm^4 = 0$. However, as recalled in \S2, $R$ is not exact.
	\end{enumerate}
\end{remark}

We use $\hdeg(-)$ to denote the homological degree. \begin{question}When does the inequality
\(
\m(R) \leq \hdeg(R)
\)
hold?\end{question}

\begin{remark}
	Let $(R, \fm, k)$ be an analytically irreducible local ring with infinite residue field. Then
	\(
	\m(R) \leq \hdeg(R).
\)
\end{remark}

\begin{proof}
	If $R$ is equi-characteristic, this is proved in \cite{sch}. So suppose $\Char R = 0 \neq p = \Char k$. We apply the same argument as in \cite{sch}. Since both $\m(R)$ and $\hdeg(R)$ behave well under completion, we may assume $R$ is a complete local domain.
	Let $\underline{x} := p, y_2, \ldots, y_d$ be a system of parameters for $R$. By Cohen?s structure theorem, there exists a regular local subring $(A, \fn) \subset R$ such that $\underline{x}A = \fn$ and $R$ is a finite $A$-module. In fact, $A = V[[y_2, \ldots, y_d]]$, where $V$ is the coefficient ring of $R$, which is a discrete valuation ring.
	By \cite[Proposition 4.1]{vas}, we have $\mu_A(R) \leq \hdeg(R)$. Therefore:
	\[
	\m(R) \leq \ell(R/\underline{x}R) = \ell(R/\fn R) = \mu_A(R) \leq \hdeg(R),
	\]
	as claimed.
\end{proof}
\begin{proposition}\label{thm:bound}
	Let \((R,\mm,k)\) be Noetherian local with infinite residue field, and
	let \(M\ne0\) be finite. Then
	\(
	{ \e(M)\le\m(M)\le\HG(M). }
	\)
	More generally \(\m(M)\le D(M)\) for any cohomological degree \(D\)
	with the standard finite-length additivity, generic hyperplane descent
	and Cohen--Macaulay calibration axioms.
\end{proposition}
\begin{proof}
	We prove the upper bound using those axioms. If \(\dim M=0\), then
	\(D(M)=\ell(M)=\m(M)\). For positive dimension put
	\(L=H^0_{\mm}(M)\) and \(N=M/L\). The module \(N\) has positive depth.
	Choose \(x\in\mm\) generic for \(N\), regular on \(N\) and a parameter
	element. The multiplication diagram for
	\(0\to L\to M\to N\to0\) gives
$
	0\to L/xL\to M/xM
	\to N/xN\to0.
$	 
	Finite-length additivity and generic descent imply
	\begin{align*}
		D(M/xM)&=D(N/xN)+\ell(L/xL)\\
		&\le D(N)+\ell(L)=D(M).
	\end{align*}
	Here additivity for any finite-length submodule follows immediately by
	applying the \(H^0\) axiom to it and to the quotient. Iterating on the
	successive quotients yields a system of parameters
	\(x_1,\ldots,x_s\) for which
	\[
	\ell\bigl(M/(x_1,\ldots,x_s)M\bigr)
	=D\bigl(M/(x_1,\ldots,x_s)M\bigr)\le D(M).
	\]
	This proves the upper bound for \(D\), in particular for \(\HG\).
	Equivalently one may iterate Gunston's Proposition 2.1.4 or 2.2.2
	directly; the proof makes the standard consequence explicit.
	For the lower bound, any parameter ideal \(q\) on \(M\) satisfies
 $	\e_{\mm}(M)\le \e_q(M)\le\ell(M/qM).
$
	The first inequality follows from \(q\subseteq\mm\); the second is
	the parameter multiplicity inequality. Taking the minimum proves it.
\end{proof}

.

\begin{proposition}\label{prop:finite}
	For
	\(
	R=\mathbb F_2[[x,y]]/(xy(x+y)),
	\)
	one has \(\m(R)=4\) and \(\HG(R)=\e(R)=3\).
\end{proposition}
\begin{proof}
	This is a one-dimensional hypersurface, hence Cohen--Macaulay. Its
	defining equation has order three, so \(\e(R)=3\), and Cohen--Macaulay
	calibration gives \(\HG(R)=3\). The parameter \(f=x+y^2\) gives
	\[
	R/fR\cong\mathbb F_2[[y]]/(y^4(1+y)),
	\]
	of length four. Thus \(\m(R)\le4\).
	Conversely, let \(f\) be any parameter. If \(f\in\mm^2\), then
	\(\ell(R/\mm^3)=1+2+3=6\). The image of the principal ideal \(fR\)
	in \(R/\mm^3\) has dimension at most one, so
	\(\ell(R/fR)\ge\ell(R/(f,\mm^3))\ge5\).
	If \(f\notin\mm^2\), its nonzero linear part over \(\mathbb F_2\)
	is one of \(x,y,x+y\). Linear coordinate changes permute these three
	forms and preserve \(xy(x+y)\). We may therefore take the linear part
	to be \(x\). The formal implicit function theorem writes
	\((f)=(x-g(y))\) for \(g\in(y^2)\). Since \(f\) is a parameter,
	\(g\ne0\), and
	\[
	\ell(R/fR)=\ord_y\bigl(g(y)y(g(y)+y)\bigr)
	=\ord_y g+2\ge4.
	\]
	This covers every parameter, including nonhomogeneous ones. The lower
	bound four and the exhibited parameter prove the result.
\end{proof}
\begin{acknowledgement}
I thank
Junzo Watanabe for his valuable comments.

\end{acknowledgement}

\end{document}